\documentclass{amsart}
\usepackage{latexsym}
\newtheorem{theorem}{Theorem}[section]
\newtheorem{lemma}[theorem]{Lemma}
\newtheorem{proposition}[theorem]{Proposition}
\newtheorem{corollary}[theorem]{Corollary}
\theoremstyle{definition}                               

\newtheorem*{example*}{Example}

\theoremstyle{remark} 

\newtheorem*{remark*}{Remark}
\numberwithin{equation}{section}

\DeclareMathOperator{\diag}{diag}

\def\R{\mathcal{R}}
\def\C{\mathcal{C}}

\def\C{\mathcal{C}}

\def\c{\mathbb{C}}

\begin{document}
\title{Equivariant maps between Calogero-Moser spaces}
\author{George Wilson}
\address{Mathematical Institute, 24--29 St Giles, Oxford OX1 3LB, UK}
\email{wilsong@maths.ox.ac.uk}
%
%
\begin{abstract}
We add a last refinement to the results of \cite{BW1} and 
\cite{BW2} relating ideal classes of the Weyl algebra to 
the Calogero-Moser varieties: we show that the bijection constructed 
in those papers is {\it uniquely determined} by its equivariance with 
respect to the automorphism group of the Weyl algebra.
\end{abstract}
\maketitle
\section{Introduction and statement of results}

Let $\, A \,$ be the Weyl algebra 
$\, \c\langle x,y \rangle /(xy - yx -1) \,$, and 
let $\, \R \,$ be the space of noncyclic right ideal classes of $\, A \,$ 
(that is, isomorphism classes of noncyclic finitely generated 
rank 1 torsion-free right $\,A$-modules). Let 
$\, \C \,$ be the disjoint union of the 
{\it Calogero-Moser spaces} $\, \C_n \,$ ($\, n \geq 1 \,$): 
we recall that $\, \C_n \,$ is the space 
of all simultaneous conjugacy classes of pairs of $\, n \times n \,$ 
matrices $\, (X,Y) \,$ such that $\, [X,Y] + I \,$ has rank $1 \,$.  
It is a smooth irreducible affine variety of dimension $\, 2n \,$ 
(see \cite{W}).
For simplicity, in what follows  we shall use the same notation 
$\, (X,Y) \,$ for 
a pair of matrices and for the corresponding point of $\, \C_n \,$. 
Let $\,G \,$ be the group of $\c$-automorphisms of 
$\, A \,$, and let $\, \Gamma \,$ and $\, \Gamma' \,$ be the isotropy 
groups of the generators $\, y \,$ and $\, x \,$ of $\, A \,$. Thus  
$\, \Gamma \,$ consists of all automorphisms of the form
$$
\Phi_p(x) = x - p(y) \,, \quad \Phi_p(y) = y 
$$
where $\, p \,$ is a polynomial; and similarly $\, \Gamma' \,$ consists of 
all automorphisms of the form
$$
\Psi_q(x) = x \,, \quad \Psi_q(y) = y - q(x) \,
$$
where $\, q \,$ is a polynomial.  According to Dixmier (see \cite{D}), 
$\, G \,$ is generated by the subgroups $\, \Gamma \,$ and $\, \Gamma' \,$. 
There is an obvious action of $\, G \,$ on $\, \R \,$; we let $\, G \,$ act 
on $\, \C \,$ by the formulae
\begin{equation}
\Phi_p(X,Y) = (X + p(Y),\, Y) \,, \quad 
\Psi_q(X,Y) = (X,\, Y + q(X)) \,.
\end{equation}
According to \cite{BW1} this $ \, G $-action is {\it transitive} 
on each space $\, \C_n \,$. The main result of \cite{BW1} was 
the following.
\begin{theorem}
\label{bw}
There is a bijection between the spaces $\, \R \,$ and $\, \C \,$ which 
is equivariant with respect to the above actions of $\, G \,$.
\end{theorem}

This bijection constructed in \cite{BW1} was obtained in a quite different 
way in \cite{BW2}.  The proof in \cite{BW2} that the two constructions 
agree used the fact that equivariance was known in both cases; thus 
to prove that the bijections coincide, it was enough to check one point 
in each $\, G$-orbit, that is, in each space $\, \C_n \,$.  The result 
to be proved in the present note is that even this (not difficult) check 
was unnecessary.
\begin{theorem}
\label{preT}
There is only one $\, G$-equivariant bijection between the spaces 
$\, \R \,$ and $\, \C \,$.
\end{theorem}

Clearly, it is equivalent to show that there is no nontrivial 
$\, G$-equivariant bijection from $\, \C \,$ to itself. We shall show  
a little more, namely, that (apart from the identity) 
there is no $\, G$-equivariant map (for short: $G$-map) at all 
from $\, \C \,$ to itself.  Since a $\, G$-map 
must take each orbit onto another orbit, that amounts to the following 
assertion.
\begin{theorem}
\label{T}
{\rm (i)}  For any $\, n \geq 1 \,$, 
let $\, f : \C_n \to \C_n \,$ be a $\, G$-map.  
Then $\, f \,$ is the identity. \\
{\rm (ii)}  For $\, n \neq m \,$ 
there is no $\, G$-map from 
$\, \C_n \,$ to $\, \C_m \,$.
\end{theorem}

Since $\, \C_n \,$ and 
the action of $\, G \,$ on it are defined by simple formulae involving 
matrices, the proof of Theorem~\ref{T} 
is just an exercise in linear algebra. Quite 
possibly there is a simpler solution to the exercise than the one given below.

The first part of Theorem~\ref{T} is equivalent to the statement 
that the isotropy group of 
any point of $\, \C \,$ (or $\, \R \,$) coincides with its normalizer in 
$\, G \,$ (see section~\ref{trivial} below); in particular, these isotropy 
groups are not normal in $\, G \,$, confirming a suspicion of Stafford 
(see \cite{St}, p.~636).  
Stafford's conjecture seems to have been the motivation for Kouakou's 
work \cite{K}, which contains a result equivalent to ours. 
The proof in \cite{K} looks quite different from the present one, because 
Kouakou does not use the spaces $\, \C_n \,$, but rather the alternative 
description of $\,\R \,$ (due to Cannings and Holland, see \cite{CH}) as the 
adelic Grassmannian of \cite{W}. I have not entirely succeeded in following 
the details of \cite{K}; in any case, it seems  
worthwhile to make available the independent verification of the result 
offered here.

\begin{remark*}
We have excluded from $\, \R \,$ the cyclic ideal class, corresponding to 
the Calogero-Moser space $\, \C_0 \,$ (which is a point).  
The reason is very trivial: since there is 
always a map from any space to a point, part (ii) of Theorem~\ref{T} would 
be false if we included $\, \C_0 \,$. However, Theorem~\ref{preT} would 
still be true.
\end{remark*}

\section{Proof of Theorem~\ref{T} in the case $\, n < m \,$}

If we accept (cf.\ \cite{BW1}, section 11) 
that the $\, \C_n \,$ are homogeneous spaces for the 
(infinite-dimensional) {\it algebraic} group $\, G \,$, then 
Theorem~\ref{T} becomes   
obvious in the case  $\, n < m \,$.  Indeed, any $\,G$-map 
from $\, \C_n \,$ to $\, \C_m \,$ would have to be a 
surjective map of {\it algebraic varieties}, which is clearly impossible  
if $\, n < m \,$, because then 
$\, \C_m \,$ has greater dimension ($2m$) than $\, \C_n \,$.  For 
readers who are not convinced by this argument, we offer a more elementary one, 
based on the following lemma.
\begin{lemma}
\label{diag}
Let $\, f : \C_n \to \C_m \,$ be a $\, G$-map.  
Suppose that $\, f(X,Y) = (P,Q) \,$, and that 
$\, P \,$ is diagonalizable. Then every eigenvalue of $\, P \,$ 
is an eigenvalue of $\, X \,$. 
\end{lemma}
\begin{proof}
Let $\, \chi \,$ be the minimum polynomial of $\, X \,$: then in 
$\, \C_m \,$ we have
$$
(P,Q) = f(X,Y) = f(X,Y + \chi(X)) = (P, Q + \chi(P))
$$
(where the last step used the fact that $\, f \,$ has to commute 
with the action of $\, \Psi_{\chi} \in G \,$).  
That means that there is an invertible matrix $\, A \,$ such that 
$$
APA^{-1} = P \text{\ \ and \ } AQA^{-1} = Q + \chi(P) \ .
$$
We may assume that $\, P = \diag(p_1, \ldots, p_m) \,$ is diagonal. 
Then since the $\, p_i \,$ are distinct (see \cite{W}, Proposition 1.10), 
$\, A \,$ is diagonal too, so 
taking the diagonal entries in the last equation gives
$\, q_{ii} = q_{ii} + \chi({p_i}) \,$, whence 
$\, \chi({p_i}) = 0 \text{\, for all \,} i \,$.
Thus $\, \chi(P) = 0 \,$, so the 
minimum polynomial of $\, P \,$ divides $\,\chi \,$.  The lemma follows.
\end{proof}
\begin{corollary}
If $\, n < m \,$ there is no $\, G$-map 
$\, f : \C_n \to \C_m \,$.
\end{corollary}
\begin{proof}
Choose $\, (P,Q) \in \C_m \,$ with $\, P \,$ diagonalizable.  Since 
$\, \C_m \,$ is just one $\, G$-orbit, $\, f \,$ is surjective, so we 
can choose $\, (X,Y) \in \C_n \,$ with $\, f(X,Y) = (P,Q) \,$.  But then 
Lemma~\ref{diag} says that $\, X \,$ is an $\, n \times n \,$ matrix 
with more than $\,n \,$ distinct eigenvalues, which is impossible.
\end{proof}

\section{The base-point}
\label{base}

A useful subgroup of $\, G \,$ 
is the group $\, R \,$ of {\it scaling transformations}, defined by
$$
R_{\lambda}(x) = \lambda x \,, \ R_{\lambda}(y) = \lambda^{-1} y 
\quad (\lambda \in \c^{\times}) \,.
$$
It acts on $\, \C_n \,$ in a similar way:
\begin{equation}
R_{\lambda}(X,Y) =  (\lambda^{-1} X,\, \lambda Y) \,.
\end{equation}
\begin{lemma}
\label{scale}
Suppose that the  conjugacy class $\, (X,Y) \in \C_n \,$ is fixed by
the group $\, R \,$.  Then $\, X \,$ and $\, Y \,$ are both nilpotent.
\end{lemma}
\begin{proof}
Let $\, \mu \,$ be an eigenvalue of (say) $\, Y \,$.  Then 
for any $\, \lambda \in \c^{\times} \,$,  
$\, \lambda \mu \,$ is an eigenvalue of $\, \lambda Y \,$, which is 
(by hypothesis) conjugate to $\, Y \,$.  Thus  
$\, \lambda \mu \,$ is an eigenvalue of $\, Y \,$ for every 
$\, \lambda \in \c^{\times} \,$, which is impossible unless $\, \mu = 0 \,$. 
Hence all eigenvalues of $\, Y \,$ must be $\, 0 \,$, that is, 
$\, Y \,$ must be nilpotent.  The same argument applies to $\, X \,$.
\end{proof}

The converse to Lemma~\ref{scale} is also true, but we shall use that 
fact only for the pair $\, (X_0, Y_0) \,$ given by
\begin{equation}
X_0 \,=\, \left(
\begin{array}{ccccc}
0 & 0 &  0 & \ldots & 0\\
1 & 0 &  0 & \ldots & 0\\
0 & 2 &  0 & \ddots & \vdots \\
\vdots & \vdots & \ddots & \ddots & 0 \\
0 & 0 & \ldots & n-1 & 0
\end{array}
\right)\ , \quad
Y_0 \, = \, \left(
\begin{array}{ccccc}
0 &  1 &     0  & \ldots & 0\\
0 &    0  &   1 & \ldots & 0\\
0 &    0  &    0  & \ddots & \vdots\\
\vdots & \vdots & \ddots & \ddots & 1 \\
0 & 0  & \ldots & 0 & 0
\end{array}
\right)\ .
\end{equation}
We shall regard $\, (X_0, Y_0) \,$ as the {\it base-point} in $\, \C_n \,$.
In the rather trivial case $\, n=1 \,$, we have $\, \C_1 = \c^2 $, and 
we interpret $\, (X_0, Y_0) \,$ as $\, (0,0) \,$.
\begin{lemma}
\label{scal}
The (conjugacy class of) the 
pair $\, (X_0, Y_0)  \in \C_n \,$ is fixed by the group $\, R \,$.
\end{lemma}
\begin{proof}
For $\, \lambda \in \c^{\times} \,$, let $ \, d(\lambda) \,$ be the diagonal 
matrix
$$
d(\lambda) := \diag(\lambda, \lambda^2, \ldots, \lambda^n) \,.
$$
Then $\, d(\lambda)^{-1} X d(\lambda) = \lambda^{-1} X \,$ 
and $\, d(\lambda)^{-1} Y d(\lambda) = \lambda Y \,$.
\end{proof}
\begin{corollary}
\label{cor}
Let $\, f : \C_n \to \C_m \,$ be a $\, G$-map, and let 
$\, f(X_0, Y_0) = (P,Q) \,$.  Then $\, P \,$ and $\, Q \,$ are nilpotent.
\end{corollary}
\begin{proof}
This follows at once from Lemmas~\ref{scale} and \ref{scal}, since 
a $\,G$-map must respect the fixed point set of any subgroup 
of $\, G \,$. 
\end{proof}

\section{Proof of Theorem~\ref{T} in the case $\, n > m \,$}

The remaining parts of the proof use the following trivial fact.
\begin{lemma}
\label{ppp}
Let $\, (X,Y) \in \C_n \,$, let $\, p \,$ be any polynomial, and let 
$\, \chi \,$ be divisible by the minimum 
polynomial of $\, X + p(Y) \,$.  Then the 
automorphism $\, \Phi_{-p} \Psi_{\chi} \Phi_p \,$ fixes $\, (X,Y) \,$.
\end{lemma}
\begin{proof}
Since $\, \chi(X + p(Y)) = 0 \,$ we have 
\begin{eqnarray}
\Phi_{-p} \Psi_{\chi} \Phi_p(X,\, Y) 
&=& \Phi_{-p} \Psi_{\chi}(X + p(Y),\, Y) \nonumber \\
&=& \Phi_{-p}(X + p(Y),\, Y) \nonumber \\
&=& (X, \,Y)\ , \nonumber
\end{eqnarray}
as claimed.
\end{proof}
\begin{proposition}
\label{n>m}
If $\, n > m > 0 \,$ there is no $\, G$-map 
$\, f : \C_n \to \C_m \,$.
\end{proposition}
\begin{proof}
We apply  Lemma~\ref{ppp} to the base-point $\, (X_0,Y_0) \in \C_n \,$, 
with $\, p(t) = t^{n-1} \,$.  The minimum ($=$ characteristic) 
polynomial of $\, X_0 + Y_0^{n-1} \,$ is 
\begin{equation}
\label{char}
\chi(t) := \det(tI - X_0 - Y_0^{n-1}) = t^n - (n-1)! \ .
\end{equation}
Now suppose that $\, f : \C_n \to \C_m \,$ is a $\, G$-map, 
and let $\, f(X_0,Y_0) = (P,Q) \,$: according to Corollary~\ref{cor}, 
$\, P \,$ and $\, Q \,$ are 
nilpotent.  They are of size less than $\, n \,$, 
so we have $\, P^{n-1} = Q^{n-1} = 0 \,$. 
Thus
\begin{eqnarray}
\Phi_{-p} \Psi_{\chi} \Phi_p(P,\, Q) 
&=& \Phi_{-p} \Psi_{\chi}(P + Q^{n-1},\, Q) \nonumber \\
&=& \Phi_{-p} \Psi_{\chi}(P,\, Q) \nonumber \\
&=& \Phi_{-p}(P,\, Q + P^n - (n-1)! I) \nonumber \\
&=& \Phi_{-p}(P,\, Q - (n-1)! I) \nonumber \\
&=& (\text{something},\, Q - (n-1)! I)\ . \nonumber
\end{eqnarray}
Now, $\, Q - (n-1)!I \,$ is not conjugate 
to $\, Q \,$ (because their eigenvalues are different), 
hence $\, \Phi_{-p} \Psi_{\chi} \Phi_p \,$ 
does not fix $\, (P,Q) \,$.  So by Lemma~\ref{ppp}, the isotropy 
group of $\, (X_0,Y_0) \,$ is not contained in the isotropy 
group of $\, f(X_0,Y_0) \,$.  This contradiction shows that $\, f \,$ 
does not exist.
\end{proof}

\section{Proof of Theorem~\ref{T} in the case $\, n = m \,$}

It remains to show that there is no nontrivial $\, G$-map from 
$\, \C_n \,$ to itself. Note that because $\, \C_n \,$ is a single orbit, 
any such map must be bijective, and must map each point of $\, \C_n \,$ 
to a point with {\it the same} isotropy group.
In the case $\, n=1 \,$ the result follows (for 
example) from Lemma~\ref{diag}, so from now on we shall assume that 
$\, n \geq 2 \,$.  Let  
$\, f : \C_n \to \C_n \,$ be a $\, G$-map, and let 
$\, f(X_0, \, Y_0) = (P,Q) \,$. Again,  Corollary~\ref{cor} says 
that $\, P \,$ and $\, Q \,$ are nilpotent. 
We aim to show that $\, (P,Q) \,$ 
can only be $\, (X_0, \, Y_0) \,$, whence $\, f \,$ is the identity. 
We remark first 
that if $\, Q^{n-1} = 0 \,$, then the calculation in the proof of 
Proposition~\ref{n>m} still gives a contradiction; thus the Jordan 
form of $\, Q \,$ consists of just one block, so we may assume that 
$\, Q = Y_0 \,$. Now, it is not hard to classify all the points 
$\, (X,\, Y_0) \in \C_n \,$ with $\, X \,$ nilpotent (see \cite{W}, 
p.26 for the elementary argument): there are exactly $\, n \,$ of them, 
and they all have the form $\, (X(\boldsymbol{a}),\, Y_0) \,$, where 
$\, \boldsymbol{a} := (a_1, \ldots, a_{n-1}) \,$ and 
$\, X(\boldsymbol{a}) \,$ denotes the subdiagonal matrix 
\begin{equation}
X(\boldsymbol{a}) \,=\, \left(
\begin{array}{ccccc}
0 & 0 &  0 & \ldots & 0\\
a_1 & 0 &  0 & \ldots & 0\\
0 & a_2 &  0 & \ddots & \vdots \\
\vdots & \vdots & \ddots & \ddots & 0 \\
0 & 0 & \ldots & a_{n-1} & 0
\end{array}
\right)\ .
\end{equation}
The possible vectors $\, \boldsymbol{a} \,$ that 
give points of $\, \C_n \,$ are
\begin{equation}
\label{a}
\boldsymbol{a} = (1,2, \ldots, r-1; -(n-r), \ldots, -2, -1)
\quad \text{for} \quad 1 \leq r \leq n \,
\end{equation}
(so $\, r=n \,$ gives $\, X_0 \,$).  Thus so far we have shown that 
$\, f(X_0, \,Y_0) \,$ must be one of these points 
$\, (X(\boldsymbol{a}),\, Y_0) \,$.  To finish the argument, we need 
the following easy calculations of characteristic polynomials 
(the first generalizes \eqref{char}):
\begin{equation}
\label{1}
\det(tI - X(\boldsymbol{a}) - Y_0^{n-1}) \,=\, 
t^n - \prod_1^{n-1} a_i
\ ;%
\end{equation}
\begin{equation}
\label{2}
\det(tI - X(\boldsymbol{a}) - Y_0^{n-2}) \,=\, 
t^n - (\prod_1^{n-2} a_i + \prod_2^{n-1} a_i)\, t \ ,
\end{equation}
where the last formula holds only for $\, n \geq 3 \,$.
If $\, \boldsymbol{a} \,$ is one of the vectors 
\eqref{a} with $\, 1 < r < n \,$, then the right hand side of 
\eqref{2} is just $\, t^n \,$; that is, 
$\, X(\boldsymbol{a}) + Y_0^{n-2} \, $ is nilpotent. In fact it is easy 
to check that the pair $\, (X(\boldsymbol{a}) + Y_0^{n-2}, \,Y_0) \,$
is conjugate to  $\, (X(\boldsymbol{a}), \,Y_0) \,$; that is, the 
map $\, (X, Y) \mapsto (X + Y^{n-2}, Y) \,$ fixes 
$\, (X(\boldsymbol{a}), \,Y_0) \,$.  It does not fix $\, (X_0, \, Y_0) \,$, 
so $\, f(X_0, \, Y_0) \,$ cannot be any of these points 
$\, (X(\boldsymbol{a}), \,Y_0) \,$.  It remains only to see that $\, f \,$ 
cannot map $\, (X_0, \, Y_0) \,$ to the pair 
corresponding to $\, r=1 \,$ in \eqref{a}:
let us call it $\, (X_1, \, Y_0) \,$.  

If $\, n \,$ is {\it even} we use \eqref{1}:  the characteristic 
polynomial of $\, X_0 + Y_0^{n-1} \,$ is $\, \chi(t) := t^n - (n-1)! \,$
while the characteristic 
polynomial of $\, X_1 + Y_0^{n-1} \,$ is $\, t^n + (n-1)! \,$, so that 
$\, \chi(X_1 + Y_0^{n-1}) = -2(n-1)! I \,$. We now apply 
Lemma~\ref{ppp} with $\, p(t) = t^{n-1} \,$. According to that lemma, 
the map 
$\, \Phi_{-p} \Psi_{\chi} \Phi_p \,$ fixes $\, (X_0, \, Y_0) \,$; 
on the other hand
\begin{eqnarray}
\Phi_{-p} \Psi_{\chi} \Phi_p(X_1,\, Y_0) 
&=& \Phi_{-p} \Psi_{\chi}(X_1 + Y_0^{n-1},\, Y_0) \nonumber \\
&=& \Phi_{-p}(X_1 + Y_0^{n-1}, \, Y_0 - 2(n-1)! I) \nonumber \\
&=& (\text{something},\, Y_0 - 2(n-1)! I) \ . \nonumber
\end{eqnarray}
Since $\, Y_0 - 2(n-1)! I \,$ is not conjugate to $\, Y_0 \,$, 
this shows that $\, \Phi_{-p} \Psi_{\chi} \Phi_p \,$ does not fix 
$\, (X_1, \, Y_0) \,$.  Thus in this case $\, f(X_0, \, Y_0) \,$ 
cannot be equal to $\, (X_1, \, Y_0) \,$

Finally, if $\, n \,$ is {\it odd}, we have a 
similar calculation using \eqref{2}.  Setting 
$\, \alpha := (n-1)! + (n-2)! \,$, the characteristic polynomial 
of $\, X_0 + Y_0^{n-2} \,$ is $\, \chi(t) := t^n - \alpha t \,$
while the characteristic 
polynomial of $\, X_1 + Y_0^{n-2} \, $ is $\, t^n + \alpha t \,$,
so that 
$\, \chi(X_1 + Y_0^{n-2}) = -2 \alpha (X_1 + Y_0^{n-2}) \,$. 
We now apply 
Lemma~\ref{ppp} with $\, p(t) = t^{n-2} \,$. The map 
$\, \Phi_{-p} \Psi_{\chi} \Phi_p \,$ fixes $\, (X_0, \, Y_0) \,$; 
on the other hand
\begin{eqnarray}
\Phi_{-p} \Psi_{\chi} \Phi_p(X_1,\, Y_0) 
&=& \Phi_{-p} \Psi_{\chi}(X_1 + Y_0^{n-2},\, Y_0) \nonumber \\
&=& \Phi_{-p}(X_1 + Y_0^{n-2}, \, Y_0 - 2 \alpha (X_1 + Y_0^{n-2})) \nonumber \\
&=& (\text{something},\, Y_0 - 2 \alpha (X_1 + Y_0^{n-2})) \ . \nonumber
\end{eqnarray}
The matrix $\, Y_0 - 2 \alpha (X_1 + Y_0^{n-2}) \,$ is not 
nilpotent, for example because its square does not have trace zero.  
Hence $\, \Phi_{-p} \Psi_{\chi} \Phi_p \,$ does not fix 
$\, (X_1, \, Y_0) \,$, and the proof is finished.
\section{Other formulations of Theorem~\ref{T}}
\label{trivial}
The remarks in this section are at the level of ``groups acting on sets'':
that is, we may as well suppose that 
$\, \R  \,$ denotes any set acted 
on by a group $\, G \,$.  We are interested in the condition
\begin{equation}
\label{simple}
\text{there is no nontrivial $\,G$-map \ } f : \R \to \R \ 
\end{equation}
(``nontrivial'' means ``not the identity map'').  As we observed above, that is 
equivalent to the two conditions
\begin{subequations}
\label{simp1}
\begin{equation}
\label{simpi}
\text{each $\,G$-orbit in $\R$ satisfies \eqref{simple}};
\end{equation}

\vspace*{-10mm}

\begin{equation}
\label{simpii}
\text{ if  $\, O_1 \,$ and $\, O_2 \,$ are distinct orbits, there is no 
$\,G$-map from $\, O_1 \,$ to $\, O_2 \,$.}
\end{equation}
\end{subequations}
%
%
Let us reformulate these conditions in terms of the isotropy groups 
$\, G_M \,$ of the points $\, M \in \R \,$. 
If $\, H \,$ and $\, K \,$ are subgroups of $\, G \,$, then any $\, G$-map 
from $\, G/H \,$ to  $\, G/K \,$ to must have the form 
$\,  \varphi(gH) = g(xK) \,$ for some $\, x \in G \,$.  This is well-defined 
if and only if we have 
$$
x^{-1}Hx \subseteq K \ .
$$ 
In the case $\, H = K \,$, that says that $\, x \in N_G(H) \,$, where 
$\, N_G \,$ denotes the normalizer in $\, G \,$: it follows that the 
$\, G$-maps from $\, G/H \,$ to itself correspond  1--1 to the points of 
$\, N_G(H)/H \,$.  Thus the conditions \eqref{simp1} are equivalent to 
\begin{subequations}
\label{simp2}
\begin{equation}
\label{simpa}
\text{for any $\,M \in \R \,$}, \text{\, we have \,} G_M = N_G(G_M)\ ;
\end{equation} 

\vspace*{-10mm}

\begin{equation}
\label{simpb}
\text{if \,} M \text{\, and \,} N  \text{\, are on different orbits, 
no conjugate of \,} G_M 
\text{ \,is in \,} G_N \ . 
\end{equation}
\end{subequations}
Finally, we note that the conditions \eqref{simp2} are equivalent to the 
single assertion
\begin{equation}
\label{simp3}
\text{if \,} G_M \subseteq  G_N \,, \text{ then \,} M = N\ . 
\end{equation}
Indeed, suppose \eqref{simp3} holds, and let $\, x \in N_G(G_M) \,$, 
that is, $\, xG_Mx^{-1} \subseteq G_M \,$, or $\, G_{xM} \subseteq G_M \,$. 
By \eqref{simp3}, we then have $\, xM = M \,$, that is, $\, x \in G_M \,$. 
Thus  \eqref{simp3} $\Rightarrow$ \eqref{simpa}.  Now, if \eqref{simpb} 
is false, we have $\, xG_Mx^{-1} \subseteq G_N \,$, that is, 
$\, G_{xM} \subseteq G_N \,$, for some $\, x \in G \,$ and some 
$\, M,N \,$ on different orbits.  But since they are on different orbits, 
$\, xM \neq N \,$, so \eqref{simp3} is false. Thus 
\eqref{simp3} $\Rightarrow$ \eqref{simpb}.

Conversely, suppose \eqref{simp2} holds, and let $\, M,N \,$ be such that 
$\, G_M \subseteq G_N \,$.  By \eqref{simpb}, $\, M \,$ and $\, N \,$ are 
on the same orbit, so $\, M = xN \,$ for some $\, x \in G \,$; hence 
$\, G_{M} = xG_Nx^{-1} \subseteq G_N \,$.  Thus $\, x \in N_G(G_N) \,$, 
so by \eqref{simpa}, $\, x \in G_N \,$: hence $\, M=N \,$, as desired.

It is in the form \eqref{simp3} that our result is stated in \cite{K}.

\scriptsize
\noindent
\textbf{Acknowledgments}.  I thank M.\ K.\ Kouakou for kindly allowing me 
see his unpublished work \cite{K}.  The main part of this paper was written 
in 2006, when I was a participant in the programme on Noncommutative 
Geometry at the Newton Institute, Cambridge; the support of EPSRC Grant 
531174 is gratefully acknowledged.
\normalsize

%

\bibliographystyle{amsalpha}

\end{document}